\documentclass{article}
\usepackage{amssymb}
\usepackage{amsfonts}
\usepackage{amsmath}

\newtheorem{theorem}{Theorem}

\newtheorem{corollary}[theorem]{Corollary}

\newtheorem{definition}[theorem]{Definition}
\newtheorem{example}[theorem]{Example}

\newtheorem{lemma}[theorem]{Lemma}

\newtheorem{proposition}[theorem]{Proposition}
\newtheorem{remark}[theorem]{Remark}

\newenvironment{proof}[1][Proof]{\noindent\textbf{#1.} }{\ \rule{0.5em}{0.5em}}
\input{tcilatex}
\begin{document}

\title{Some topological properties of one dimensional cellular automata}
\author{Rezki Chemlal. \\
%EndAName
Laboratory of applied mathematics.\\
Faculty of exact sciences \and University of Bejaia, Algeria.}
\maketitle

\begin{abstract}
Cellular automata, CA for short are continuous maps defined on the set of
configurations over a finite alphabet $A$ that commutes with the shift. They
are characterized by the existence of local function which determine by
local behavior the image of an element of the configurations space.We will
study periodic factors of cellular automata. We classify periodic factors
according to their periods. We show that for a surjective cellular automaton
with equicontinuity points but without being equicontinuous there is an
infinity of classes of equivalence of periodic factors.
\end{abstract}

\section*{Introduction}

Cellular automata are continuous maps commuting with the Bernoulli shift.
They are used to model complex systems, for example in statistical physics
as microscale models of hydrodynamics \cite{Chop03}, to model social choices 
\cite{Nowak96} or tumor growth \cite{ABM03}.

The study of CA as symbolic dynamical systems began with Hedlund \cite{Hed69}%
. Dynamical behavior of cellular automata is studied mainly in the context
of discrete dynamical systems by equipping the space of configurations with
the product topology which make it homeomorphic to the Cantor space.

Any equicontinuous dynamical system on a zero dimensional space is almost
periodic. An equicontinuous surjective cellular automaton is periodic \cite%
{BT00}.

We will classify periodic factors of surjective equicontinuous cellular
automata according to their periods. Two periodic dynamical systems with the
same period will belong to the same equivalence class.

We show that any cellular automaton with equicontinuity points has a
periodic factor. Furthermore any surjective cellular automaton with
equicontinuity points is either periodic or has an infinity of classes of
periodic factors.

As a corollary the maximal equicontinuous factor of a surjective cellular
automaton with equicontinuous points but without being equicontinuous is not
a cellular automaton.

The paper is divided into two parts, part one is for basic definition and
background, part two is for new results.

\section{Definitions and background.}

\subsection{Topological dynamics.}

A topological dynamical system $\left( X,T\right) $ consists of a compact
metric space $X$ and a continuous self--map $T$ .

A point $x$ is said periodic if there exists $p>0$ with $T^{p}\left(
x\right) =x.$ The least $p$ with this property is called the period of $x.$
A\ point $x$ is eventually or ultimately periodic if $T^{m}\left( x\right) $
is periodic for some $m\geq 0.$

In the same way a dynamical system is said periodic if there exists $p>0$
with $T^{p}\left( x\right) =x$ for every $x\in X$ and eventually periodic if 
$T^{m}$ is periodic for some $m\geq 0.$

A point $x\in X$ is said to be an equicontinuity point, or to be Lyapunov
stable, if for any $\epsilon >0$, there exists $\delta >0$ such that if $%
d\left( x,y\right) <\delta $ one has $d\left( T^{n}\left( y\right)
,T^{n}\left( x\right) \right) <\epsilon $ for any integer $n\geq 0$.

When all points of $\left( X,T\right) $ are equicontinuity points $\left(
X,T\right) $ is said to be equicontinuous: since $X$ is compact an
equicontinuous system is uniformly equicontinuous. We say that $\left(
X,T\right) $ is almost equicontinuous if the set of equicontinuous points is
residual.

We say that $\left( X,T\right) $ is sensitive if for any $x\in X$ we have :%
\begin{equation*}
\exists \epsilon >0,\forall \delta >0,\exists y\in B_{\delta }\left(
x\right) ,\exists n\geq 0\text{ }\mathrm{such}\text{ \textrm{that }}d\left(
T^{n}\left( y\right) ,T^{n}\left( x\right) \right) \geq \epsilon .
\end{equation*}

We say that $\left( X,T\right) $ is expansive if we have :%
\begin{equation*}
\exists \epsilon >0,\forall x\neq y\in X,\exists n\geq 0,d\left( T^{n}\left(
x\right) ,T^{n}\left( y\right) \right) \geq \epsilon .
\end{equation*}

A dynamical system $\left( X,T\right) $\emph{\ }is transitive if for any
nonempty open sets $U,V\subset A^{\mathbb{Z}}$ there exists $n>0$ with $%
U\cap F^{-n}\left( V\right) \neq \emptyset .$ It is said topologically
mixing if for any nonempty open sets $U,V\subset A^{\mathbb{Z}},U\cap
F^{-n}\left( V\right) \neq \emptyset $ for all sufficiently large $n.$

\subsection{Factors}

Let $(X,T)$ and $(Y,U)$ two dynamical systems and $\pi $ a continuous
surjective map $\pi :X\rightarrow Y$ such that $\pi \circ T=U\circ \pi .$We
say that $\pi $ is a factor map and $(Y,U)$ is a factor of $(X,T).$If $\pi $
is bijective, we say that $\pi $ is a conjugacy and that $(X,T)$ and $(Y,U)$
are conjugate. The conjugacy is an equivalence relation on dynamical systems.%
\begin{equation*}
\begin{tabular}{lll}
$X$ & $\overset{T}{\longrightarrow }$ & $X$ \\ 
$\pi \downarrow $ &  & $\downarrow \pi $ \\ 
$Y$ & $\overset{U}{\longrightarrow }$ & $Y$%
\end{tabular}%
\end{equation*}

The factor relation is transitive in the sense that if $\left( X,T\right)
,\left( Y,U\right) ,\left( Z,V\right) $ are three dynamical systems such
that $\left( Y,U\right) $ is a factor of $\left( X,T\right) $ and $\left(
Z,V\right) $ is a factor of $\left( Y,U\right) $\ then $\left( Z,V\right) $
is a factor of $\left( X,T\right) .$

From the definition the factor of a surjective dynamical system is still
surjective. If $\left( X,T\right) $ is periodic of period $p$ and $\left(
Y,U\right) $ is a factor of $\left( X,T\right) $ are two dynamical systems
such that $\left( Y,U\right) $ is a factor of $\left( X,T\right) $. If $%
\left( X,T\right) $ and $\left( Y,U\right) $ are periodic of respectively
periods $p$ and $q$ then $q$ is a divisor of $p.$

\begin{proposition}
Any dynamical system $\left( X,T\right) $ possesses a maximal equicontinuous
factor $\pi :\left( X,F\right) \rightarrow \left( Y,U\right) ,$ such that
for any equicontinuous factor $\varphi :\left( X,T\right) \rightarrow \left(
Z,V\right) $ there is a unique factor map $\psi :\left( Y,U\right)
\rightarrow \left( Z,V\right) $ with $\psi \circ \pi =\varphi .$ The system $%
\left( Y,U\right) $ is unique up to conjugacy.
\end{proposition}

\subsection{Symbolic dynamics}

Let $A$ be a finite set called the \emph{alphabet}. A word is a any finite
sequence of elements of $A$.Denote by $A^{\ast }=\cup _{n\in \mathbb{N}%
^{\ast }}A^{n}$ the set of all finite words $u=u_{0}...u_{n-1};$the length
of a word $u\in A^{n}$ is $\left\vert u\right\vert =n.$

Let $A^{\mathbb{Z}}$ denote the set of all functions $x:\mathbb{Z\rightarrow 
}A,$ which we regard as $\mathbb{Z}$-indexed \emph{configurations} of
elements in $A$.

We write such a configuration as $x=\left( x_{n}\right) _{n\in \mathbb{Z}},$
where $x_{n}\in A$ for all $n\in \mathbb{Z}$, and refer to $A^{\mathbb{Z}}$
as \emph{configuration space}.

Treat $A$ as a discrete topological space; then A is compact, so $A^{\mathbb{%
Z}}$ is compact in the Tychonoff product topology. In fact, $A^{\mathbb{Z}}$
is a Cantor space: it is compact, perfect,totally disconnected, and
metrizable.

The standard metric on $A^{\mathbb{Z}}$ is defined by

\begin{equation*}
d\left( x,y\right) =2^{-n}with\text{ }n=mini\geq 0:x_{i}\neq y_{i}\,\mathrm{%
or}\,x_{-i}\neq y_{-i}
\end{equation*}

Let $x$ a configuration of $A^{\mathbb{Z}},$ for two integers $i,j$ with $%
i<j $ we denote by $x\left( i,j\right) \in A^{j-i+1}$ the word $%
x_{i}...x_{j}.$

For any word $u$ we define the cylinder $\left[ u\right] _{l}=\left\{ x\in
A^{\mathbb{Z}}:x\left( l,l+\left\vert u\right\vert -1\right) =u\right\} $
where the word $u$ is at the position $l.$ The cylinder $\left[ u\right]
_{0} $ is simply noted $\left[ u\right] $. The cylinders are clopen (closed
open) sets.

The metric $d$ is non archimedian; that is $d\left( x,y\right) \leq \max
\left\{ d\left( x,z\right) ,d\left( z,y\right) \right\} $. Consequently any
two cylinders either one contain the other or they intersect trivially.

The shift map $\sigma :$ $A^{\mathbb{Z}}\rightarrow $ $A^{\mathbb{Z}}$ is
defined as $\sigma \left( x\right) _{i}=x_{i+1},$ for any $x\in A^{\mathbb{Z}%
}$ and $i\in \mathbb{Z}$. The shift map is a continuous and bijective
function on $A^{\mathbb{Z}}.$ The dynamical system $\left( A^{\mathbb{Z}%
},\sigma \right) $ is commonly called \emph{full shift, }it is mixing and
have a dense set of periodic points.

The configuration $^{\infty }u^{\infty }$ is defined by $\left( ^{\infty
}u^{\infty }\right) _{k\left\vert u\right\vert +i}=u_{i}$ for $k\in \mathbb{Z%
}$ for $k\in \mathbb{Z}$, $0\leq i<\left\vert u\right\vert $ and $u\in
A^{\ast }.$ The configuration $^{\infty }u^{\infty }$ is shift periodic and
is called \emph{spatially periodic} configuration.

\subsection{Cellular automata}

A cellular automaton is a continuous map $F:A^{\mathbb{Z}}\rightarrow A^{%
\mathbb{Z}}$ if there exists integers $m\leq a$ (memory and anticipation)
and a local rule $f:A^{a-m+1}\rightarrow A$ such that for any $x\in A^{%
\mathbb{Z}}$ and any $i\in \mathbb{Z}$ 
\begin{equation*}
F\left( x\right) _{i}=f\left( x_{\left[ i+m,i+a\right] }\right)
\end{equation*}%
We call $r=\max \left\{ -m,a\right\} $ the radius of $F$ and $d=a-m\geq 0$
its diameter.

Cellular automata can be topologically characterized as morphisms of the
full shift \cite{Hed69}.

A cellular automaton is surjective iff any nonempty word has exactly $\left(
\#A\right) ^{d}$ preimages.

If a cellular automaton is not surjective then there exists a diamond i.e. a
word $w\in A^{d}$ and distinct finite words $u,v$ of same length such that $%
f\left( wuw\right) =f\left( wvw\right) .$ As a corollary any injective
cellular automaton is surjective.

\subsubsection{Equicontinuity and Kurka's classification}

\begin{definition}
Let $F$ be a cellular automaton. A word $w$ with $\left\vert w\right\vert
\geq s$ is an $s$-blocking word for $F$ if there exists $p\in \left[
0,\left\vert w\right\vert -s\right] $ such that for any $x,y\in \left[ w%
\right] $ we have $F^{n}\left( x\right) \left( p,p+s\right) =F^{n}\left(
y\right) \left( p,p+s\right) $ for all $n\geq 0.$
\end{definition}

\begin{proposition}[\protect\cite{Kur03}]
Let $F$ be a cellular automaton\ with radius $r>0.$ The following conditions
are equivalent.\newline
1. $F$ is not sensitive. \newline
2. $F$ has an $r-$blocking word.\newline
3. $F$ has some equicontinuous point.\newline
4. There exists $k>0$ such that every word $u\in A^{2k+1}$ is an $r$%
-blocking word.\newline
5. $F$ is almost periodic.
\end{proposition}

Notice that every equicontinuous dynamical system on a zero dimensional
space is eventually periodic.

\begin{proposition}
Let $\left( A^{\mathbb{Z}},F\right) $ be a surjective cellular automaton of
radius $r$ with equicontinuity points. If $w$ is a $r-$blocking word; then
for any word $u$ between two occurrence of the word $w$ the sequel of the
word $u$ is eventually periodic.
\end{proposition}

\begin{proof}
Let $w$ be a blocking word then there exist $p$ $\in \left[ 0,r-s\right] $
such that :%
\begin{equation*}
\forall x,y\in \left[ w\right] _{0},\forall n\geq 0,F^{n}\left( x\right)
\left( p,p+r\right) =F^{n}\left( y\right) \left( p,p+r\right) .
\end{equation*}%
For any word $u$ the word $wuw$ is also a blocking word.

Hence all elements from the cylinder $\left[ wuw\right] _{0}$ have same
images between coordinates $p+r$ and $\left\vert w\right\vert +\left\vert
u\right\vert +p+r.$ Let us consider the shift periodic point $\left(
wuw\right) ^{\infty }\in \left[ wuw\right] _{0},$which is eventually
periodic for $F$ consequently the sequence $\left\{ F^{k}\left( wuw\right)
\left( p,\left\vert w\right\vert +\left\vert u\right\vert +p+r\right) ,k\in 
\mathbb{N}\right\} $ is also eventually periodic.
\end{proof}

\begin{proposition}[\protect\cite{Kur97}]
Any equicontinuity point of a cellular automaton $\left( A^{\mathbb{Z}%
},F\right) $ has an occurrence of a blocking word. Conversely if there exist
blocking words, any point with infinitely many occurrences of a blocking
word to the left and right is an equicontinuity point; if moreover $X$ is
transitive for the shift equicontinuity points are dense in $A^{\mathbb{Z}}.$
\end{proposition}

\begin{proposition}[\protect\cite{BT00}]
\label{AcEqut} Let $F$ be an equicontinuous surjective cellular automaton;
there exists then $p>0$ such that $F^{p}=Id$; in particular $F$ is bijective.
\end{proposition}

\subsubsection{Gilman's classification}

Based on the Wolfram's work \cite{Wol84}, Gilman \cite{Gil87}\cite{Gil88}
introduced a measure theoretic classification using Bernoulli measures. We
will use only topological part of this classification one can refer to the
references given above for a more deep insight into this classification.

Let $F$ be a cellular automaton. For any interval $\left[ i_{1},i_{2}\right] 
$ the relation $\mathfrak{R}$ defined by $x\mathfrak{R}y$ if and only if $%
\forall j:F^{j}\left( x\right) \left( i_{1},i_{2}\right) =F^{j}\left(
y\right) \left( i_{1},i_{2}\right) $ is an equivalence relation.

The equivalence class of a point $x$ will be denoted by : 
\begin{equation*}
B_{\left[ i_{1},i_{2}\right] }\left( x\right) =\left\{ y\in A^{\mathbb{Z}%
},\forall j:F^{j}\left( x\right) \left( i_{1},i_{2}\right) =F^{j}\left(
y\right) \left( i_{1},i_{2}\right) \right\} .
\end{equation*}

If we visualize the behavior of ($f$ with argument) $x$ as an array with
entry $\left( a_{ij}\right) $ with entry $a_{ij}=\left( f^{i}\left( x\right)
\right) \left( j\right) $ in row $i$ and column $j,$ then $B_{\left[
i_{1},i_{2}\right] }\left( x\right) $ is the set of all $y$ whose behavior
agrees with that of $x$ on infinite vertical strip under the interval $\left[
i_{1},i_{2}\right] .$ More generally we will say that $y$ has behavior $B_{%
\left[ i_{1},i_{2}\right] }\left( x\right) $ on $\left[ i_{1}+k,i_{2}+k%
\right] $ if $\sigma ^{-k}\left( y\right) \in B_{\left[ i_{1},i_{2}\right]
}\left( x\right) .$

\begin{lemma}[Gilman ]
If $x\in A^{\mathbb{Z}}$ and $n\in \mathbb{N}^{\ast }$ then\newline
$\left( i\right) $ The set $B_{\left[ -n,n\right] }\left( x\right) $ is
closed.\newline
$\left( ii\right) $ We have $F\left( B_{\left[ -n,n\right] }\left( x\right)
\right) \subset B_{\left[ -n,n\right] }\left( F\left( x\right) \right) .$%
\newline
$\left( iii\right) $ $F$ is equicontinuous at $x$ if and only if $x\in B_{%
\left[ -m,m\right] }\left( x\right) ^{\circ }$ for all $m\in \mathbb{N}%
^{\ast }.$\newline
$\left( iv\right) $ The restriction of $F$ to $\overline{O_{x}}$ is
equicontinuous if and only if $B_{\left[ -m,m\right] }\left( x\right) $ is
ultimately periodic for all $m\in \mathbb{N}^{\ast }.$
\end{lemma}

\begin{proposition}[Gilman]
Let $F$ be a cellular automaton of radius $r.$The following are equivalent 
\newline
$\left( i\right) $ $F$ is equicontinuous at some point $x\in A^{\mathbb{Z}}.$%
\newline
$\left( ii\right) $ For some $n\geq \frac{r-1}{2}$ there is a class $B_{%
\left[ -n,n\right] }\left( x\right) $ with $B_{\left[ -n,n\right] }\left(
x\right) ^{\circ }\neq \varnothing .$\newline
$\left( iii\right) $ For all $n\geq 0$ there is class $B_{\left[ -n,n\right]
}\left( x\right) $ with $B_{\left[ -n,n\right] }\left( x\right) ^{\circ
}\neq \varnothing .$
\end{proposition}

\subsection{Periodic points of cellular automata}

Let $(A^{\mathbb{Z}},F)$ be a cellular automaton. By commutation with the
shift, every shift-periodic point is $F-$eventually periodic hence the set
of eventually periodic points is dense.

Shift periodic points are called spatially periodic and periodic points of a
cellular automata that are not shift periodic are called strictly temporally
periodic.

Boyle and Kitchens \cite{BK99} showed that closing cellular automata have a
dense set of periodic points. The same result was obtained by Blanchard and
Tisseur \cite{BT00} for surjective cellular automata with equicontinuity
points.

Acerbi, Dennunzio and Formenti \cite{ADF09} showed that if every mixing
cellular automaton has a dense set of periodic points then every cellular
automaton has a dense set of periodic points.

The question whatever a surjective cellular automaton has a dense set of
periodic points is still an open problem.

\section{Results}

This section is for new results. We will focus on the set of periodic
factors of a surjective cellular automaton with equicontinuous points but
without being equicontinuous.

This set contain the set of factors that are surjective equicontinuous
cellular automata. We classify periodic factors according to their periods

\begin{definition}
We define the equivalence relation between two periodic dynamical systems $%
\left( X,F\right) $ and $\left( Y,G\right) $ as : $F\sim G$ if and only iff $%
F$ and $G$ have the same period.
\end{definition}

We will denote by $\widetilde{\mathit{p}}$ the equivalence class of periodic
dynamical systems with period $p.$

We will show that a cellular automaton with equicontinuous points but
without being equicontinuous has an infinity of classes of periodic factors.

The following lemma and its proof is a topological version of an ergodic
result show by Gilman \cite{Gil88}.

\begin{lemma}
\label{Gil}Let $(A^{\mathbb{Z}},F)$ be a cellular automaton or radius $r$,
If $F$ has an equicontinuity point $x$ then the sequence of words $\left(
F^{i}\left( x\right) \left( -r,r\right) \right) _{i\geq 0}$ is eventually
periodic.
\end{lemma}

\begin{proof}
Let $r$ be the radius of the cellular automaton and let $x$ be an
equicontinuous point: Then we have $B_{\left[ -r,r\right] }\left( x\right)
^{\circ }\neq \varnothing $.\newline
As the shift is mixing there exist an integer $p>0$ such that :%
\begin{equation*}
\sigma ^{-p}\left( B_{\left[ -r,r\right] }\left( x\right) \right) \cap
\left( B_{\left[ -r,r\right] }\left( x\right) \right) \neq \varnothing .
\end{equation*}

For any $y\in \sigma ^{-p}\left( B_{\left[ -r,r\right] }\left( x\right)
\right) \cap B_{\left[ -r,r\right] }\left( x\right) $ we have:%
\begin{equation*}
\forall i\geq 0:F^{i}\left( y\right) \left( -r,r\right) =F^{i}\left(
y\right) \left( -r+p,r+p\right) .
\end{equation*}

Let us denote $y\left( -r,r\right) =w$ and by $u$ the word between two
occurrences of $w$. Let us also denote by $y^{-}$ the part at left of the
first word $w$ and $y^{+}$ that's at right of the second $w.$

By incorporating the word $uw$ in $y$ between the words $w$ and $u$ we
obtain a new element $y^{\left( 1\right) }$ containing two occurrences of
the word $uw.$ By repeating this process we obtain a sequence $y^{\left(
i\right) }$ containing at each iteration one more occurrence of the word $uw.
$

By recurrence under the $F-$ action we can show that that $y^{\left(
i\right) }$ still belongs to $B_{\left[ -ij_{1}-\left( i+1\right)
r,ij_{1}+\left( i+1\right) r\right] }\left( x\right) $ for any $i>0.$ 
\newline
For $i=1$ we show that $y^{\left( 1\right) }\in B_{\left[ -j_{1}+2r,j_{1}+2r%
\right] }\left( x\right) .$\newline
As images of the word $y^{\left( 1\right) }\left( j_{1},j_{1}+2r\right) =w$
are independent at right from the infinite column under the word $y^{+}$ and
at left from the infinite column under $wu$. From another side images of the
word $y^{\left( 1\right) }\left( -j_{1}-2r,-j_{1}\right) =w$ depend at left
from the infinite column under $y^{-}$ and at right from the column under $%
wu $.

From the definition of the set $B_{\left[ -r,r\right] }\left( x\right) ,$
images of $x\left( -r,r\right) =w$ are independent from the behavior at left
and right of $w.$So we have : 
\begin{equation*}
\forall i:F^{i}\left( y^{\left( 1\right) }\right) \left( -r,r\right)
=F^{i}\left( y^{\left( 1\right) }\right) \left( -r,r\right) .
\end{equation*}%
\newline
Then we have: $y^{\left( 1\right) }\in B_{\left[ -j_{1}+2r,j_{1}+2r\right]
}\left( x\right) .$

Suppose now that $y^{\left( i\right) }\in B_{\left[ -ij_{1}-\left(
i+1\right) r,ij_{1}+\left( i+1\right) r\right] }\left( x\right) $ we want to
show that $y^{\left( i+1\right) }\in B_{\left[ -\left( i+1\right)
j_{1}-\left( i+2\right) r,\left( i\right) i+1j_{1}+\left( i+2\right) r\right]
}\left( x\right) .$\newline
Same arguments as step $i=1$ lead to conclude that images of the word $%
y^{\left( i\right) }\left( ij_{1},ij_{1}+\left( i+1\right) r\right) =w$
depends at right from the infinite word $y^{+}$ and at left from the word $%
wu $ and from the infinite column below. On another hand images of the word $%
y^{\left( 1\right) }\left( -ij_{1},-ij_{1}-\left( i+1\right) r\right) =w$
depends at left from the infinite word $y^{-}$ and at right from $wu$ and
the infinite column below.\newline
As each element $y^{\left( i\right) }$ belongs to $B_{\left[ -ij_{1}-\left(
i+1\right) r,ij_{1}+\left( i+1\right) r\right] }\left( x\right) $ it belongs
also to $B_{\left[ -r,r\right] }\left( x\right) .$

The sequence of configurations $y^{\left( i\right) }=y^{-}w\underset{%
\leftarrow \text{ }\left( i-1\right) \text{ }\rightarrow }{uw...uw}y^{+}$
containing at each step a one more occurrence of the word $uw$ converge to
the shift periodic configuration $\left( uw\right) ^{\infty }$ which share
with $y$ the same coordinates over $\left( -r,r\right) $.\newline
As the set $B_{\left[ -r,r\right] }\left( x\right) $ is closed the periodic
point $\left( uw\right) ^{\infty }$ is in $B_{\left[ -r,r\right] }\left(
x\right) ,$ the sequence of words $\left( F^{i}\left( x\right) \left(
-r,r\right) \right) _{i\geq 0}$ is then eventually periodic.
\end{proof}

\begin{proposition}
Let $(A^{\mathbb{Z}},F)$ be a cellular automaton with equicontinuous points
then $F$ has as factor at least a periodic factor.
\end{proposition}

\begin{proof}
As $F$ has some equicontinuous point then there is a class $B_{\left[ -n,n%
\right] }\left( x\right) $ with $B_{\left[ -n,n\right] }\left( x\right)
^{\circ }\neq \varnothing .$\newline

By lemma \ref{Gil} the sequence of words $\left( F^{i}\left( x\right) \left(
-r,r\right) \right) _{i\geq 0}$ is eventually periodic, then there exist a
preperiod $m$ and a period $p$ such that for all $i\in \mathbb{N}:$ 
\begin{equation*}
F^{m}\left( x\right) \left( -r,r\right) ^{\infty }=F^{m+ip}\left( x\right)
\left( -r,r\right) ^{\infty }.
\end{equation*}%
\newline
Let us define the family of sets $W_{k}=\left[ F^{k}\left( x\right) \left(
-r,r\right) \right] $ with $0\leq k\leq m+p-1.$

Consider the set $W=\cup _{m}^{m+p-1}W_{k}$ we will show that the
restriction of $F$ to $W$ has a periodic factor.

Let the function $\pi $ defined from $W$ to $\mathbb{Z}/p\mathbb{Z}$ by :%
\begin{equation*}
\pi \left( x\right) =k-m:x\in W_{k}:m\leq k\leq m+p-1
\end{equation*}%
Define the periodic dynamical systems $\left( \mathbb{Z}/p\mathbb{Z},\left(
x+1\right) \func{mod}p\right) ,$ we have 
\begin{equation*}
\pi \left( F\left( x\right) \right) =\left\{ 
\begin{array}{l}
k-m+1\text{ if }x\in W_{k}\text{ and }m\leq k\leq m+p-2 \\ 
1\text{if }x\in W_{m+p-1}%
\end{array}%
\right. =P\left( \pi \left( x\right) \right)
\end{equation*}%
Thus $\left( \mathbb{Z}/p\mathbb{Z},\left( x+1\right) \func{mod}p\right) $
is a periodic factor of $\left( W,F\right) .$
\end{proof}

\begin{example}
Non surjective CA with equicontinuous points \newline
$\left( \mathbf{3}^{\mathbb{Z}},F\right) $ where $\mathbf{3}=\left\{
w,0,r\right\} $ the rules are defined by 
\begin{equation*}
\begin{tabular}{|l|l|l|l|l|l|l|l|l|l|}
\hline
$x_{i-1}x_{i}$ & $wr$ & $w0$ & $ww$ & $rr$ & $r0$ & $rw$ & $0r$ & $00$ & $0w$
\\ \hline
$F\left( x\right) _{i}$ & ~~$r$ & ~~$0$ & ~~$w$ & ~$0$ & ~$r$ & ~$w$ & ~$0$
& ~$0$ & ~$w$ \\ \hline
\end{tabular}%
\end{equation*}%
Notice that there is no word of lentgh 3 such that we can obtain the word $%
rr $ so the cellular automaton is not surjective.\newline
We have also : 
\begin{equation*}
F\left( \left[ w000w\right] \right) =F\left( \left[ w00rw\right] \right) =%
\left[ w000w\right]
\end{equation*}%
\newline
Hence $\left( \mathbf{3}^{\mathbb{Z}},F\right) $ is not injective.\newline
Consider the cylinder $\left[ wr000w\right] $ we have 
\begin{equation*}
\begin{tabular}{l|l|llllll|l}
$x$ & $...$ & $w$ & $r$ & $0$ & $0$ & $0$ & $w$ & $...$ \\ 
$F\left( x\right) $ & $...$ & $w$ & $r$ & $r$ & $0$ & $0$ & $w$ & $...$ \\ 
$F^{2}\left( x\right) $ & $...$ & $w$ & $r$ & $0$ & $r$ & $0$ & $w$ & $...$
\\ 
$F^{3}\left( x\right) $ & $...$ & $w$ & $r$ & $r$ & $0$ & $r$ & $w$ & $...$
\\ 
$F^{4}\left( x\right) $ & $...$ & $w$ & $r$ & $0$ & $r$ & $0$ & $w$ & $...$%
\end{tabular}%
\end{equation*}%
\newline
Denote by $W=\left[ wr0r0w\right] \cup \left[ wrr0rw\right] $ then the
restriction of $F$ to the set $W$ has a periodic factor.
\end{example}

The following lemma is a particular case of a result shown by Tisseur \cite%
{Tis08}.

\begin{lemma}
Let $(A^{\mathbb{Z}},F)$ be a surjective cellular automaton with some
equicontinuous point $x.$ If we denote by $r$ the radius of the cellular
automaton then every point $y\in A^{\mathbb{Z}}$ is a limit of the sequence
of the $F$-periodic points of the form $\left( x_{\left[ -r,r\right] }y_{%
\left[ -n,n\right] }x_{\left[ -r,r\right] }\right) ^{\infty }.$
\end{lemma}

\begin{remark}
The proof of the last lemma uses arguments from ergodic theory. To use tools
from ergodic theory we need an invariant measure. In the case of cellular
automata we know that the uniform measure is invariant iff the cellular
automaton is surjective. For more details the reader can look to the
following references \cite{Hed69},\cite{Gil87},\cite{Gil88},\cite{Tis08}.
\end{remark}

\begin{proposition}
Let $(A^{\mathbb{Z}},F)$ be a surjective cellular automaton with
equicontinuous points but without being equicontinuous; then the set of
equicontinuous factors contain an infinite union of equivalence classes $%
\widetilde{\mathit{p}}$.
\end{proposition}

\begin{proof}
Let us suppose that the set of equicontinuous factors contain a finite
number of equivalence classes $\widetilde{\mathit{p}}$.\newline
Denote by $\widetilde{\mathit{p}}_{1},\widetilde{\mathit{p}}_{2},...%
\widetilde{\mathit{p}}_{s}$ the existing $s$ equivalence classes.

Let $r$ be the radius of the cellular automaton and let $x$ be an
equicontinuous point.

Let $y\in A^{\mathbb{Z}}$ from the proof of Lemma \ref{Gil} we know that the
sequence 
\begin{equation*}
F^{i}\left( x\left( -r,r\right) y\left( -n,n\right) x\left( -r,r\right)
\right)
\end{equation*}
is eventually periodic hence associated to a periodic factor.

Let $P=\underset{1\leq i\leq s}{\func{lcm}}\left( \widetilde{\mathit{p}}%
_{i}\right) $ and the sequence of the $F-$periodic points \newline
$\left( x\left( -r,r\right) y\left( -n,n\right) x\left( -r,r\right) \right)
^{\infty }$ that converges to $y.$

We have then for any $y$ in $A^{\mathbb{Z}}$ : 
\begin{eqnarray*}
F^{P}\left( y\right) &=&F^{P}\left( \underset{k\rightarrow \infty }{\lim }%
x\left( -r,r\right) y\left( -n,n\right) x\left( -r,r\right) ^{\infty }\right)
\\
&=&\underset{k\rightarrow \infty }{\lim }F^{P}\left( x\left( -r,r\right)
y\left( -n,n\right) x\left( -r,r\right) ^{\infty }\right) \\
&=&\underset{n\rightarrow \infty }{\lim }\left( x\left( -r,r\right) y\left(
-n,n\right) x\left( -r,r\right) ^{\infty }\right) =y.
\end{eqnarray*}%
\newline
We obtain then $F^{P}=Id$ and hence $F$ is equicontinuous which is a
contradiction.
\end{proof}

\begin{corollary}
The maximal equicontinuous factor of a surjective cellular automaton with
equicontinuity points but without being equicontinuous is not a cellular
automaton.
\end{corollary}

\begin{proof}
Suppose that the maximal equicontinuous factor is a cellular automaton and
denote it by $M$. It is then surjective and there exist a period $p$ such
that $M^{p}=M.$

As there is an infinity of surjective periodic factors classes. Choose a
periodic factor $E$ such that $q$ the period of $E$ satisfy $q>p$ as $E$ is
a factor of $M$ then $q$ is a divisor of $p$ which is a contradiction.
\end{proof}

\section{Conclusion}

Two results are shown about periodic factors of cellular automata. A
surjective cellular automaton with equicontinuity points but without being
equicontinuous has an infinity of periodic factors. As every surjective
equicontinuous cellular automaton is periodic it comes that the maximal
equicontinuous factor is not a cellular automaton. An interesting direction
is to try to characterize the maximal equicontinuous factor of a cellular
automaton.

Another problem indirectly connected to the question of the existence of
equicontinuous factors is that of transitive cellular automata. We know that
a transitive cellular automaton is weakly mixing. It is still an open
problem to know if weakly mixing implies mixing or not.

\end{document}